\documentclass{amsart}
\usepackage{amsaddr}
\usepackage[english]{babel}

\usepackage{fullpage}
\usepackage{geometry}
\usepackage[dvipsnames]{xcolor}
\usepackage{graphicx}
\usepackage{multirow}
\usepackage{wrapfig}
\usepackage{longtable}
\usepackage{float}
\usepackage{setspace} 
\usepackage{array}
\usepackage{booktabs,tabularx}
\usepackage{amsmath,amssymb,enumerate,url,amsthm}
\usepackage{ytableau}
\usepackage[numbers,sort]{natbib}
\usepackage{subcaption}
\usepackage{gensymb}
\usepackage{enumitem}
\usepackage{pgf,tikz, varwidth}
\usetikzlibrary{arrows, matrix}
\usetikzlibrary{shapes.symbols, calc, shapes, decorations}
\usepackage{chngcntr}
\usepackage{hyperref,cases}

\usepackage[normalem]{ulem}

\usepackage[bottom]{footmisc}

\definecolor{dblue}{RGB}{0,54,135}
\definecolor{lorange}{RGB}{255,155,0}

\newcommand{\cp}{\mathrm{cp}}
\newcommand{\CP}{\mathcal{CP}}

\newcommand{\N}{\mathbb{N}}

\theoremstyle{definition}
\newtheorem*{remark}{Remark}
\newtheorem{theorem}{Theorem}
\newtheorem{cor}[theorem]{Corollary}
\newtheorem{lemma}[theorem]{Lemma}
\newtheorem{proposition}[theorem]{Proposition}
\newtheorem*{defn}{Definition}
\newtheorem{example}[theorem]{Example}
\newtheorem{conj}[theorem]{Conjecture}
\numberwithin{theorem}{section}

\allowdisplaybreaks

\title{On the positivity of infinite products connected to partitions with even parts below odd parts and copartitions}
\author{Hannah E. Burson}
\address{School of Mathematics\\ University of Minnesota, Twin Cities\\ Minneapolis, MN 55455} 
\email{hburson@umn.edu}

\author{Dennis Eichhorn}
\address{Department of Mathematics\\ University of California, Irvine\\ Irvine, CA 92697} \email{deichhor@math.uci.edu}

\date{}

\keywords{copartitions, partitions with even parts below odd parts, positivity, partitions}
\subjclass[2010]{05A17,  11P81}

\begin{document}
\maketitle

\begin{abstract}
    In this paper, we give a combinatorial proof of a positivity result of Chern related to Andrews's $\mathcal{EO}^*$-type partitions.
This combinatorial proof comes after 
reframing Chern's result in terms of copartitions.
     Using this new perspective, we also reprove an overpartition result of Chern by showing that it comes essentially ``for free" from our combinatorial proof and some basic properties of copartitions.
    Finally, the application of copartitions leads us to more general positivity conjectures for families of both infinite and finite products, with a proof in one special case.
\end{abstract}

\section{Introduction}  

In \cite{BursonEichhorn}, the authors  introduce copartitions.  Copartitions reframe and generalize Andrews's $\mathcal{EO}^*$-type partitions, which arose in a study of the combinatorics of the mock theta function $\nu(q)$ \cite{Andrews18}.
Each $(a,b,m)$-copartition is comprised of three partitions: 
\begin{itemize}
    \item a partition into parts $\equiv  a \pmod m$, which we call the ground;
    \item a partition into parts $\equiv b \pmod m$, which we call the sky; and
    \item a rectangular partition that unites them.
\end{itemize}
We denote the number of $(a,b,m)$-copartitions as $\cp_{a,b,m}(n)$.
Andrews's $\mathcal{EO}^*(n)$ counts the number of partitions of $n$ with all even parts smaller than all odd parts and only the largest even part appearing an odd number of times. We call the set of partitions counted by $\mathcal{EO}^*$ ``$\mathcal{EO}^*$-type partitions." 
A straightforward bijection reveals that $\cp_{1,1,2}(n) = \mathcal {EO}^*(2n)$, and so we may recast any $\mathcal{EO}^*$-type partition result in terms of copartitions.

This paper is motivated by a study into a weighted version of the $\mathcal{EO}^*(n)$ function completed by Chern in \cite{Chern21}. In Chern's study, he assigned a weight of $-1$ to the partition $\lambda$ if the largest even part of $\lambda$ is $2\pmod{4}$ and $1$ otherwise. 
We restate Chern's result in the language of copartitions. Let $\cp_{a,b,m}^o(n)$ (resp. $\cp_{a,b,m}^e(n)$) be the number of $(a,b,m)$-copartitions of $n$ with an odd (resp. even) number of ground parts. 
\begin{theorem}[Chern]\label{thm:inequality} 
For all $n\ge 0$,
 \begin{numcases}{\cp_{1,1,2}^o(n)} 
\le \cp_{1,1,2}^e(n) &if $n$ is even \notag\\
=\cp_{1,1,2}^e(n) &if $n$ is odd. \notag
\end{numcases}
\end{theorem}
\noindent
Chern called for a direct combinatorial proof of Theorem \ref{thm:inequality}, and we achieve this in Section \ref{sec:combinatorialproof}.

Chern also studied the overpartition analogue of Andrews's set of partitions and proved an overpartition analogue of Theorem \ref{thm:inequality}. In Section \ref{sec:overcopartitions}, we show that Chern's overpartition analogue follows ``for free" from basic properties of overpartitions and our combinatorial proof of Theorem \ref{thm:inequality}.

Both Andrews's $\mathcal{EO}^*$-type partitions and $(a,b,m)$-copartitions have generating functions that enjoy nice infinite product forms. Thus,
Chern's Theorem \ref{thm:inequality} is equivalent to the non-negativity of the coefficients of 
$$\frac{(-q^{2};q^2)_\infty}{(-q;q^2)_\infty(q;q^2)_\infty}.$$
Similarly, considering a weighted version of the generating function for $(a,b,m)$-copartitions
leads naturally to a question about the more general $q$-products

\begin{equation}
    \frac{(-q^{a+b};q^m)_\infty}{(-q^a;q^m)_\infty(q^b;q^m)_\infty} \label{genProduct}.
\end{equation}
Many researchers are interested in the positivity of families $q$-products (see \cite{StantonPositivity}),
and in Section \ref{sec:conjecture}, we consider this product and offer a conjecture.

\section{Background on Copartitions}

In this section, we recall the definition of an $(a,b,m)$-copartition and give some relevant background.  
By a partition of $n$, we mean a non-increasing finite sequence of positive integers $\lambda=(\lambda_1,\lambda_2,\lambda_3, \ldots,\lambda_r)$ such that $\lambda_1+\lambda_2+\ldots+\lambda_r=n$. We use $\nu(\lambda)=r$ to denote the number of parts of the partition $\lambda$. 

An $(a,b,m)$-copartition is a triple of partitions $(\gamma,\rho,\sigma)$ such that all parts in $\gamma$ are congruent to $a\pmod{m}$, all parts in $\sigma$ are congruent to $b\pmod{m}$, and $\rho$ consists of exactly $\nu(\sigma)$ parts of size $m\cdot \nu(\gamma)$. To handle the cases where $a$ or $b$ is larger than $m$, we add the condition that the parts in $\gamma$ must have size at least $a$ and the parts in $\sigma$ must have size at least $b$.  We call $\gamma$ the \emph{ground} of the copartition and $\sigma$ the \emph{sky}.  Note that one can represent an $(a,b,m)$-copartition graphically by appending the $m$-modular diagram for $\sigma$ to the right of the $m$-modular diagram for $\rho$ and then appending the conjugate of the $m$-modular diagram for $\gamma$ below $\rho$. 

The graphical representation allows for a straightforward conjugation map on copartitions by reflecting the diagram about the line $y=-x$. Equivalently, one can define the conjugate of the $(a,b,m)$-copartition $(\gamma, \rho,\sigma)$ as the $(b,a,m)$-copartition $(\sigma,\rho',\gamma)$, where $\rho'$ is the conjugate of the $m$-modular diagram for $\rho$. 

\begin{example}
The following diagram represents the $(a,b,m)$-copartition $((4m+a,3m+a,a,a)$, $(4m,4m)$, $(3m+b,3m+b))$ and its conjugate $((3m+b,3m+b),(2m,2m,2m,2m),(4m+a,3m+a,a,a))$.

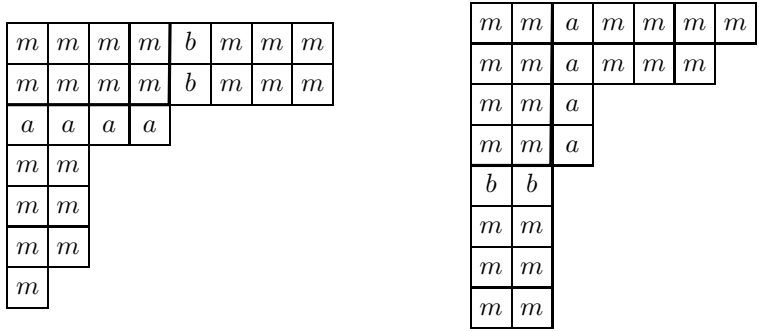
\begin{center}
\begin{minipage}{0.4\textwidth}
\begin{tikzpicture}[remember picture]

\node (n) {\begin{varwidth}{5cm}{
\begin{ytableau}
m&m&m&m&b&m&m&m\\
m&m&m&m&b&m&m&m\\
a&a&a&a\\
m&m\\
m&m\\
m&m\\
m
\end{ytableau}} \end{varwidth}};
\draw[very thick, black] ([xshift=0.4em, yshift=+2.3em] n.west)--([xshift=6.5em,yshift=2.3em]n.west)--([xshift=0em,yshift=-.4em] n.north);

\end{tikzpicture}
\end{minipage}
\begin{minipage}{0.3\textwidth}
\begin{tikzpicture}[remember picture]

\node (n) {\begin{varwidth}{5cm}{
\begin{ytableau}
m&m&a&m&m&m&m\\
m&m&a&m&m&m\\
m&m&a\\
m&m&a\\
b&b\\
m&m\\
m&m\\
m&m
\end{ytableau}} \end{varwidth}};
\draw[very thick, black] ([xshift=0.4em, yshift=+0em] n.west)--([xshift=3.4em,yshift=0em]n.west)--([xshift=-2.3em,yshift=-.4em] n.north);

\end{tikzpicture}
\end{minipage}
\end{center}
\end{example}

We also define the $\emph{enlarged sky}$ of a copartition, written as $\rho|\sigma$, by adding the $i$th part of $\sigma$ to the $i$th part of $\rho$ for each $i$. 
For example, in the $(3,1,4)$-copartition $((11,11,7),(12,12,12,12),(13,9,9,1))$, the enlarged sky is the partition $(25,21,21,13)$. Note that, because of the required dimensions of $\rho$, the smallest part of $\rho|\sigma$ must be of size at least $m$ times $\nu(\gamma)$.

The $\mathcal{EO}^*$-type partition corresponding to the $(1,1,2)$-copartition $(\gamma,\rho,\sigma)$ is obtained by creating two copies of each part of $\rho|\sigma$ and doubling each part of $\gamma'$. For example, $((7,7,3),(6,6),(5,1))$ corresponds to $(11,11,7,7,6,6,6,4,4,4,4).$ 

In \cite{BursonEichhorn}, the authors define $\cp_{a,b,m}(w,s,n)$ to be the number of $(a,b,m)$-copartitions of size $n$ that have $w$ ground parts and $s$ sky parts and prove the following infinite product form of the generating function for $\cp_{a,b,m}(w,s,n)$.
\begin{theorem}[Burson--Eichhorn\cite{BursonEichhorn}]\label{thm:sumProd}
We have 
\begin{align*}
   {\mathbf{cp}}_{a,b,m}(x,y,q)&:= \sum_{n=0}^\infty\sum_{w=0}^\infty\sum_{s=0}^\infty \cp_{a,b,m}(w,s,n)x^{s}y^{w}q^n\\
   &=\frac{(xyq^{a+b};q^m)_\infty}{(xq^b;q^m)_\infty(yq^a;q^m)_\infty}.
\end{align*}
\end{theorem}

By setting $x=1$ and $y=-1$ in Theorem \ref{thm:sumProd}, we immediately get the following corollary, which interprets \eqref{genProduct} as a difference of copartition functions. 
\begin{cor}
For $a,b,m \in \N$ and $n \in \N_0$,
\begin{equation}\label{eq:weightedProd}
    \sum_{n=0}^\infty (\cp_{a,b,m}^e(n)-\cp_{a,b,m}^o(n))q^n=\frac{(-q^{a+b};q^m)_\infty}{(-q^a;q^m)_\infty(q^b;q^m)_\infty}.
\end{equation}
\end{cor}

\section{Combinatorial proof of Theorem \ref{thm:inequality}}\label{sec:combinatorialproof}
In this section, we give a combinatorial proof of Theorem \ref{thm:inequality} by describing an injection $\phi:\CP_{1,1,2}^o(n)\to \CP_{1,1,2}^e(n),$ where $\CP_{1,1,2}^o(n)$ (resp. $\CP_{1,1,2}^e(n)$) is the set of copartitions counted by the function $\cp_{1,1,2}^o(n)$ (resp. $\cp_{1,1,2}^e(n)$).  We will also show that, when $n$ is even, this injection is a bijection. 

We begin with a general outline of the injection. For $(\gamma, \rho,\sigma)\in \mathcal{CP}_{1,1,2}(n)$, we define the map $\phi:\mathcal{CP}_{1,1,2}^o(n)\to \mathcal{CP}_{1,1,2}^e(n)$ by iterating the process of removing the largest part and largest even part of $\gamma'$ (the conjugate partition of $\gamma$), and adding the sum of the removed parts to $\rho|\sigma$ as a new part. 
Note that the map is well defined when the largest part of $\gamma'$ is odd. 
However, while iterating this process, we may obtain objects that are no longer $(1,1,2)$-copartitions.  
Thus, to show the details, we begin by defining a set of pairs of partitions that slightly generalize copartitions.

Let $\CP'(n)$ denote the set of pairs of partitions  $(\tilde\gamma,\tilde\sigma)$ with total size $n$ such that all parts of $\tilde{\sigma}$ are odd, the smallest part of $\tilde{\sigma}$ is greater than the sum of the largest part of $\tilde{\gamma}$ and the largest even part of $\tilde{\gamma}$, and the only part in $\tilde\gamma$ appearing an odd number of times is either the largest part or the largest even part.
Note that, if $(\gamma,\rho,\sigma)\in \mathcal{CP}_{1,1,2}(n)$, then $(\gamma',\rho|\sigma)\in \mathcal{CP}'(n)$. Furthermore, we can write $\mathcal{CP}'(n)=\mathcal{CP}_o'(n)\cup\mathcal{CP}_e'(n)$, where $\mathcal{CP}'_o(n)$ consists of all pairs in which the largest part of $\tilde{\gamma}$ is odd and appears an odd number of times and $\mathcal{CP}_e'(n)$ consists of all pairs in which the largest even part of $\tilde{\gamma}$ (which may or may not also be the largest part of $\tilde{\gamma}$) is the only part of $\tilde{\gamma}$ appearing an odd number of times.

For $(\tilde\gamma,\tilde\sigma)\in \CP'(n)$, we define $f((\tilde\gamma,\tilde\sigma))=(\tilde\gamma_f,\tilde\sigma_f)$ to be the pair obtained by removing the largest part and the largest even part of $\tilde\gamma$, which we call $\ell(\tilde\gamma)$ and $\ell_e(\tilde{\gamma})$, respectively, and adding a part of size $\ell(\tilde\gamma)+\ell_e(\tilde\gamma)$ to $\tilde\sigma$. 
If $\tilde{\gamma}$ has no even parts, we define $\ell_e(\tilde{\gamma})=0$. Note that, when $\ell(\tilde\gamma)\ne \ell_e(\tilde\gamma)$, $f$ is size-preserving. Additionally, we define $g((\tilde\gamma,\tilde\sigma))=(\tilde\gamma_g,\tilde\sigma_g)$ to be the pair obtained by removing the smallest part from $\tilde\sigma$, which we denote as $s(\tilde\sigma)$, and adding two new parts to $\tilde\gamma$: one equal to the unique part of $\tilde\gamma$ appearing an odd number of times, which we denote as $o(\tilde\gamma)$, and one part equal to the difference between the size of the smallest part of $\tilde\sigma$ and $o(\tilde\gamma)$.

\begin{example}
Consider $(\tilde\gamma,\tilde\sigma)=(\{7,6,6,3,3,2,2,2,2\},\{15,15,13\})\in\CP'_o(76)\subset\CP'(76).$ Note that $o(\tilde\gamma)=\ell(\tilde\gamma)=7$, $\ell_e(\tilde\gamma)=6$, and $s(\tilde\sigma)=13$. Then, $f(\tilde\gamma,\tilde\sigma)=(\{6,3,3,2,2,2,2\},\{15,15,13,13\})\in \CP'_e(76)$ and $g(\tilde\gamma,\tilde\sigma)=(\{7,7,6,6,6,3,3,2,2,2,2\},\{15,15\})\in \CP'_e(76).$
\end{example}

 We now give three lemmas which may appear to be quite technical, but they simply clarify the domain and range of $f$ and $g$ and demonstrate that $g$ is the inverse of $f$. 
Note that $f$ is well-defined when  $\ell(\tilde{\gamma})\ne\ell_e(\tilde{\gamma})$ and $g$ is well-defined for all $(\tilde{\gamma},\tilde{\sigma})\in \CP'(n)$ with non-empty $\tilde{\sigma}$. Furthermore, where they are well-defined, $f$ and $g$ both change the number of appearances of $\ell(\tilde\gamma)$ and $\ell_e(\tilde\gamma)$ by exactly one, so that $o(\tilde\gamma)$ and $o(\tilde\gamma_f)$ have opposite parity. (Note that the previous statement holds even if the number of appearances of a part changes from zero to one or from one to zero). Thus, we obtain the following lemma. 
\begin{lemma}\label{lem}
\begin{enumerate}[label={(\arabic*)}]

    \item If $(\tilde{\gamma},\tilde{\sigma})\in \CP'_o(n)$, then $f((\tilde{\gamma},\tilde{\sigma}))\in \CP'_e(n)$. \label{lem:fCPotoCPe}
    \item If $(\tilde\gamma,\tilde\sigma)\in \CP_o'(n)$ with $s(\tilde\sigma)<2\ell(\tilde\gamma)$, then $g((\tilde\gamma,\tilde\sigma))\in \CP'_e(n)$ and $\ell(\tilde\gamma_g)>\ell_e(\tilde\gamma_g)$. \label{lem:gCPotoCPe}
    \item If $(\tilde\gamma,\tilde\sigma)\in \CP'_e(n)$ and $\tilde{\sigma}\ne \{\}$, then $g((\tilde\gamma,\tilde\sigma))\in \CP'_o(n)$. \label{lem:gCPetoCPo}
    \item If $(\tilde\gamma,{\tilde{\sigma}})\in \CP'_e(n)$ with $\ell(\tilde\gamma)> \ell_e(\tilde\gamma)$, then $f((\tilde\gamma,\tilde\sigma))\in \CP'_o(n)$.  \label{lem:fCPetoCPo}

\end{enumerate}
\end{lemma}

\begin{lemma}\label{lem:gcompf}
For $(\tilde\gamma,\tilde\sigma)\in \CP'(n)$ with $\ell(\tilde\gamma)>\ell_e(\tilde\gamma)$, $g\circ f((\tilde\gamma,\tilde\sigma))=(\tilde\gamma,\tilde\sigma)$.
\end{lemma}

\begin{proof}
Let $(\tilde\gamma, \tilde\sigma)\in \CP'(n)$ such that $\ell(\tilde\gamma)>\ell_e(\tilde\gamma)$.  We can decompose $\tilde\gamma$ as $\{\ell(\tilde\gamma),\ell_e(\tilde\gamma)\}\cup\widehat\gamma$. Then, we get that $f((\tilde\gamma,\tilde\sigma))=(\widehat\gamma,\tilde\sigma\cup\{\ell(\tilde\gamma)+\ell_e(\tilde\gamma)\})$. Now, we consider two cases as follows:
\begin{enumerate}
    \item[Case 1:] $(\tilde\gamma,\tilde\sigma)\in \CP'_o(n).$ In this case, we note that $o(\tilde\gamma_f)=o(\widehat\gamma)=\ell_e(\widehat\gamma)=\ell_e(\tilde\gamma)$. Thus, $g\circ f(\tilde\gamma,\tilde\sigma)=(\{\ell(\tilde\gamma),\ell_e(\tilde\gamma)\}\cup \widehat\gamma,\tilde\sigma)=(\tilde\gamma,\tilde\sigma)$.
    \item[Case 2:] $(\tilde\gamma,\tilde\sigma)\in \CP'_e(n).$ In this case, we note that $o(\tilde\gamma_f)=o(\widehat\gamma)=\ell(\widehat\gamma)=\ell(\tilde\gamma).$ Furthermore, $s(\tilde\sigma_f)-o(\tilde\gamma_f)=\ell(\tilde\gamma)+\ell_e(\tilde\gamma)-\ell(\tilde\gamma)=\ell_e(\tilde\gamma)$. Therefore, $g\circ f(\tilde\gamma,\tilde\sigma)=(\{\ell(\tilde\gamma),\ell_e(\tilde\gamma)\}\cup \widehat\gamma,\sigma)=(\tilde\gamma,\tilde\sigma).$ \qedhere
\end{enumerate}
\end{proof}
\begin{lemma}\label{lem:fcompg}
For $(\tilde\gamma,\tilde\sigma)\in \CP'_e(n)$ or $(\tilde\gamma,\tilde\sigma)\in\CP'_o(n)$ with $s(\tilde\sigma)<2\ell(\tilde\gamma)$, $f\circ g((\tilde\gamma,\tilde\sigma))=(\tilde\gamma,\tilde\sigma)$.
\end{lemma}

\begin{proof}
Let $(\tilde\gamma,\tilde\sigma)\in \CP'_e(n)$ or $(\tilde\gamma,\tilde\sigma)\in\CP'_o(n)$ with $s(\tilde\sigma)<2\ell(\tilde\gamma).$ We can decompose $\tilde\sigma$ as $\{s(\tilde\sigma)\}\cup \widehat\sigma$. Then, $g((\tilde\gamma,\tilde\sigma))=((\{s(\tilde\sigma)-o(\tilde\gamma),o(\tilde\gamma)\}\cup\tilde\gamma),\widehat\sigma)$. We consider two cases as follows:
\begin{enumerate}
    \item[Case 1:] $(\tilde\gamma,\tilde\sigma)\in \CP_o(n)$ with $s(\tilde\sigma)<2\ell(\tilde\gamma)$. Thus, $o(\tilde\gamma)=\ell(\tilde\gamma)>\ell_e(\tilde\gamma)$. Note that, because $s(\tilde\sigma)<2\ell(\tilde\gamma)$ and $o(\tilde\gamma)=\ell(\tilde\gamma),$ $s(\tilde\sigma)-o(\tilde\gamma)=s(\tilde\sigma)-\ell(\tilde\gamma)<\ell(\tilde\gamma)$. Moreover, because $s(\tilde\sigma)\ge \ell(\tilde\gamma)+\ell_e(\tilde\gamma)$, $s(\tilde\sigma)-o(\tilde\gamma)\ge \ell_e(\tilde\gamma)$. Thus, since $\ell(\tilde\gamma)$ is odd, $\ell_e(\tilde\gamma_g)=s(\tilde\sigma)-o(\tilde\gamma)=s(\tilde\sigma)-\ell(\tilde\gamma)$. Therefore, $f\circ g((\tilde\gamma,\tilde\sigma))=(\tilde\gamma,\widehat\sigma\cup\{s(\tilde\sigma)-\ell(\tilde\gamma)+\ell(\tilde\gamma)\})=(\tilde\gamma,\tilde\sigma)$.
    \item[Case 2:] $(\tilde\gamma,\tilde\sigma)\in \CP_e(n).$ Thus, $o(\tilde\gamma)=\ell_e(\tilde\gamma).$ Then, $\ell_e(\tilde\gamma_g)=\ell_e(\tilde\gamma)=o(\tilde\gamma)$ and $\ell(\tilde\gamma_g)=s(\tilde\sigma)-\ell_e(\tilde\gamma).$ Therefore, $f\circ g(\tilde\gamma,\tilde\sigma)=(\tilde\gamma,\widehat\sigma\cup\{s(\tilde\sigma)-\ell_e(\tilde\gamma)+\ell_e(\tilde\gamma)\})=(\tilde\gamma,\tilde\sigma).$ \qedhere
\end{enumerate}
\end{proof}

We are now ready to prove Theorem \ref{thm:inequality}. 

\begin{proof}
Let $n\in \N$. Note that we can recast any $(\gamma,\rho,\sigma)\in \CP_{1,1,2}^o(n)$ as $(\gamma',\rho|\sigma)\in \CP'_o(n)$ where $s(\rho|\sigma)>2\ell(\gamma')$,
and we can recast any $(\widehat\gamma,\widehat\rho,\widehat\sigma)\in \CP_{1,1,2}^e(n)$ as $({\widehat\gamma}',\widehat\rho|\widehat\sigma)\in \CP'_e(n)$ where $\ell({\widehat\gamma}')=\ell_e({\widehat\gamma}').$ 
We now explicitly give an injection
$\phi:\{(\tilde\gamma,\tilde\sigma)\in \CP'_o(n) :s(\tilde\sigma)>2\ell(\tilde\gamma)\} \to \{(\tilde\gamma,\tilde\sigma)\in \CP'_e(n):\ell(\tilde\gamma)=\ell_e(\tilde\gamma)\}$, which equivalently gives an injection  $\bar\phi:\CP_{1,1,2}^o(n)\to \CP_{1,1,2}^e(n)$ in the natural way. 

By Lemmas \ref{lem} -- \ref{lem:fcompg}, we have that $f$ and $g$ are inverse functions that are well-defined on most pairs in $\CP'(n)$. Note that each application of $f$ reduces the number of parts of $\tilde{\gamma}$ by either 1 or 2. Thus, for each $(\tilde\gamma,\tilde\sigma) \in \CP'_e(n)$ with $s(\tilde\sigma) > 2\ell(\tilde\gamma)$, letting $(\tilde\gamma_{f^k}, \tilde\sigma_{f^k}):=f^k((\tilde\gamma,\tilde\sigma))$, there must be some minimum positive integer $k$ such that either the largest part remaining in $\tilde\gamma_{f^k}$ is even or $\tilde{\gamma}_{f^k}$ is empty.
In other words,
$f^k((\tilde\gamma,\tilde\sigma))\in \CP'_e(n)$ with $\ell(\tilde\gamma_{f^k})=\ell_e(\tilde\gamma_{f^k}).$
For each $(\tilde\gamma,\tilde\sigma)\in \CP_o'(n)$ with $s(\tilde\sigma)>2\ell(\tilde\gamma),$ 
define $\phi((\tilde\gamma,\tilde\sigma)) = f^k((\tilde\gamma,\tilde\sigma))\in \CP'_e(n)$. 

By Lemmas \ref{lem:gcompf} and \ref{lem:fcompg}, we know that $f:\{(\tilde\gamma,\tilde\sigma)\in\mathcal{CP}'(n)|\ell(\tilde\gamma)> \ell_e(\tilde\gamma)\}\to \{(\tilde\delta, \tilde\tau) \in \mathcal{CP}'(n) | \tilde\tau \ne \emptyset\}$ is bijective. 
Therefore, $\phi$ is an injection. 

We now show that when our copartition has odd size, $\phi$ is a bijection.
It suffices to show that for any $(\tilde\gamma,\tilde\sigma)\in \CP'_e(n)$ with $|\tilde\gamma|+|\tilde\sigma|$ odd and $\ell(\tilde\gamma)=\ell_e(\tilde\gamma)$, there is a positive integer $k$ such that $g^k((\tilde\gamma,\tilde\sigma))\in \CP'_o(n)$ with $s(\tilde\sigma)>2\ell(\tilde\gamma).$  Let $(\tilde\gamma,\tilde\sigma) \in \mathcal{CP}'_e(n).$ Thus, $|\tilde\gamma|$ must be even. Then, because $|\tilde\gamma|+|\tilde\sigma|$ is odd and $\tilde\sigma$ has odd parts, $\nu(\tilde\sigma)$ must be odd. Thus, either $g^{\nu(\tilde\sigma)}(\tilde\gamma,\tilde\sigma)\in \CP'_o(n)$ or there is some minimum integer $j<\nu(\tilde\sigma)$ such that $g^j(\tilde\gamma,\tilde\sigma)$ is not well-defined. In the first case, our proof is complete, so we assume we are in the second case.  By Lemma \ref{lem}, parts  \ref{lem:gCPotoCPe} and \ref{lem:gCPetoCPo}, the first case where $g^j(\tilde \gamma,\tilde\sigma)$ is not well-defined is when $g^{j-1}(\tilde\gamma,\tilde\sigma)\in \CP'_o(n)$ and $s(\tilde\sigma_{g^{j-1}})\ge 2\ell(\tilde\gamma_{g^{j-1}})$. Since $s(\tilde\sigma_{g^{j-1}})$ is always odd and $2\ell(\tilde\gamma_{g^{j-1}})$ is always even, we must have $s(\tilde\sigma_{g^{j-1}})> 2\ell(\tilde\gamma_{g^{j-1}})$, as desired. 

 Since $\phi$ is an injection, and when our copartition has odd size, $\phi$ is a bijection, Theorem \ref{thm:inequality} follows.
\end{proof}

Now, we illustrate the above proof with some examples.

\begin{example}
Consider the copartition $$(\gamma,\rho,\sigma)=(\{9,9,9,9,5,5,3\},\{14,14,14\},\{5,5,3\})\in \CP_{1,1,2}^o(104).$$ Then $\gamma'=\{7,7,7,6,6,4,4,4,4\}$ and $\rho|\sigma=\{19,19,17\}$, so we work with $$(\{7,7,7,6,6,4,4,4,4\},\{19,19,17\})\in \CP'_o(104).$$

By iteratively applying $f$, we obtain the following sequence:
\begin{align*}
(\{7,7,7,6,6,4,4,4,4\},\{19,19,17\})&\xrightarrow{\quad f\quad}(\{7,7,6,4,4,4,4\},\{19,19,17,13\})\\
&\xrightarrow{\quad f\quad}(\{7,4,4,4,4\},\{19,19,17,13,13\})\\
&\xrightarrow{\quad f\quad}(\{4,4,4\},\{19,19,17,13,13,11\})\in \CP_e'(104).
\end{align*}
Thus, $f^3(\{7,7,7,6,6,4,4,4,4\},\{19,19,17\})=(\{4,4,4\},\{19,19,17,13,13,11\})$. Furthermore, we can conjugate the first partition and separate the second partition to get 
\begin{align*}
    \bar\phi((\{9,9,9,9,5,5,3\},&\{14,14,14\},\{5,5,3\}))\\&=(\{3,3,3,3\}, \{8,8,8,8,8,8\},\{11,11,9,5,5,3\}).
\end{align*}

Note that $(\{3,3,3,3\}, \{8,8,8,8,8,8\},\{11,11,9,5,5,3\})\in \CP_{1,1,2}^e(104)$.
\end{example}

\begin{example}
Here, we show an example of a copartition in $\CP_{1,1,2}^e(48)$ that is not in the range of this map. Consider $ (\{\},\{\},\{13,13,9,7,3,3\})\in \CP_{1,1,2}^e(48),$ which corresponds to $(\tilde\gamma,\tilde\sigma) = (\{\}$, $\{13,13,9,7,3,3\})$ $\in \CP'_e(n)$. 
We iteratively apply $g$ without obtaining an element $(\tilde\delta,\tilde\tau)\in \CP'_o(48)$ with $2\ell(\tilde\delta)<s(\tilde\tau)$ as follows.
\begin{align*}
(\{\},\{13,13,9,7,3,3\})&\xrightarrow{\quad g\quad} (\{3\},\{13,13,9,7,3\})\\
&\xrightarrow{\quad g\quad} (\{3,3\},\{13,13,9,7\})\\
&\xrightarrow{\quad g\quad} (\{7,3,3\},\{13,13,9\})\\
&\xrightarrow{\quad g\quad} (\{7,7,3,3,2\},\{13,13\})\\
&\xrightarrow{\quad g\quad} (\{11,7,7,3,3,2,2\},\{13\})\\
&\xrightarrow{\quad g\quad} (\{11,11,7,7,3,3,2,2,2\},\{\}).
\end{align*}
We can no longer apply $g$ and we never obtained a pair where the smallest part of the second partition is at least twice the size of the largest part of the first partition.  Thus, we could not construct an element of $\CP_{1,1,2}^o(48)$ at any point in our process, so $(\{\},\{\},\{13,13,9,7,3,3\})\notin\phi(\CP_{1,1,2}^o(48))$. 
\end{example}

\begin{remark}
We can characterize the set $\CP^e_{1,1,2}(n)\setminus \phi(\CP^o_{1,1,2}(n))$ by considering the pairs of partitions that result from iteratively applying $g$ to the pair in $\CP_e'(n)$ corresponding to any partition in $\CP^e_{1,1,2}(n)$. 
Note that, by Lemma \ref{lem}, parts \ref{lem:gCPotoCPe} and \ref{lem:gCPetoCPo}, for any $(\tilde\gamma,\tilde\sigma)\in \CP'(n)$, either 
\begin{enumerate}[label={(\arabic*)}]
    \item $(\tilde\gamma,\tilde{\sigma})\in \CP_e'(n)$ with $\tilde\sigma=\{\}$;
    \item $(\tilde\gamma,\tilde{\sigma})\in \CP_e'(n)$ and $g((\tilde\gamma,\tilde{\sigma}))\in \CP_o'(n)$;
    \item $(\tilde\gamma,\tilde\sigma)\in \CP_o'(n)$ and can be recast as a copartition in $\CP_{1,1,2}^o(n)$; or 
    \item $(\tilde\gamma,\tilde\sigma)\in \CP_o'(n)$ and $g((\tilde\gamma,\tilde\sigma))\in \CP_e'(n)$.
\end{enumerate}
Therefore, beginning from a pair in $\CP'_e(n)$ corresponding to a copartition in $\CP_{1,1,2}^e(n)$, we can iteratively apply the map $g$ until we reach a pair $(\tilde\gamma,\tilde\sigma)$ such that either $(\tilde\gamma,\tilde\sigma)$ corresponds to a copartition in $\CP_{1,1,2}^o(n)$ or $(\tilde\gamma,\tilde\sigma)\in \CP_e'(n)$ and $\tilde\sigma=\{\}$. 
A similar argument shows that iteratively applying $f$ to any pair $(\tilde\gamma,\tilde\sigma)$ such that either $(\tilde\gamma,\tilde\sigma)$ corresponds to a copartition in $\CP_{1,1,2}^o(n)$ or $(\tilde\gamma,\tilde\sigma)\in \CP_e'(n)$ and $\tilde\sigma=\{\}$ eventually results in a pair that corresponds to a copartition in $\CP^e_{1,1,2}(n)$.
Thus, $\CP^e_{1,1,2}(n)\setminus\phi(\CP_{1,1,2}^o(n))$ is equinumerous with the set of partitions $\pi$ such that $\ell_e(\pi)$ is the only part size appearing an odd number of times. 

We now obtain the generating function for such partitions by generating the even parts and the odd parts separately.
The generating function for partitions into odd parts with each part size appearing an even number of times is ${1}/{(q^2;q^4)_\infty}.$
By conjugation, we see that the number of partitions of $n$ into even parts with only the largest part appearing an odd number of times is also equal to the number of partitions of $n$ into odd parts with each part size appearing an even number of times.  Thus
the set of partitions $\pi$ such that $\ell_e(\pi)$ is the only part size appearing an odd number of times has the generating function
$$\frac{1}{(q^2;q^4)_\infty^2}=\frac{(-q^2;q^2)_\infty}{(q^2;q^4)_\infty}.$$ 
Thus, beyond proving Theorem \ref{thm:inequality}, we have obtained a fully combinatorial proof that $$\sum_{n=0}^\infty(\cp_{1,1,2}^e(n)-\cp_{1,1,2}^o(n))q^n=\frac{(-q^2;q^2)_\infty}{(q^2;q^4)_\infty}.$$
\end{remark}

\section{Overpartition Analogues}\label{sec:overcopartitions}

In \cite{Chern21}, Chern also treated an overpartition analogue of Andrews's $\mathcal{EO}^*$-type partitions.
In particular, using $q$-series techniques, he showed that his equivalent form of Theorem \ref{thm:inequality} holds for overpartitions.
\begin{theorem}[Chern]\label{thm:EOBARinequality} 
Let $\overline{\mathcal{EO}}_0^*(n)$ (resp. $\overline{\mathcal{EO}}_2^*(n)$) be the number of overpartitions of $n$ such that all even parts are smaller than all odd parts and that the largest even part is congruent to $0$ (resp. $2$) modulo $4$ and is the only part appearing an odd number of times. Then,

$$\overline{\mathcal{EO}}_0^*(n)\begin{cases} = \overline{\mathcal{EO}}_2^*(n) & \text{if $n$ is not divisible by $4$}\\
\ge \overline{\mathcal{EO}}_2^*(n) & \text{if $n$ is divisible by $4$.}
\end{cases}$$
\end{theorem}
\noindent Chern also called for a direct combinatorial proof of Theorem \ref{thm:EOBARinequality}.
By rewriting this identity in the language of $(1,1,2)$-copartitions, we now show that a combinatorial proof follows
directly
from basic properties of overpartitions along with the combinatorics of copartitions established in the previous sections.

Recall that for any set of partitions $\mathcal{P}$, the associated set of overpartitions $\overline{\mathcal{P}}$ is the set of all partitions from $\mathcal{P}$ where first occurrence (equivalently, the final occurrence) of a part may be overlined \cite{overpartitions}.
Overpartitions may also be counted as a sum over the original set of partitions.
\begin{proposition}[Corteel-Lovejoy]
Let $\mathcal{P}$ be a set of partitions and let $\overline{\mathcal{P}}$ be the associated set of overpartitions.
Then 
\begin{equation}\label{eq:overformula}
|\overline{\mathcal{P}}| = \sum_{\lambda \in \mathcal{P}} 2^{dv(\lambda)},
\end{equation}
where $dv(\lambda)$ denotes the \emph{diversity}, or number of different part sizes, of the partition $\lambda$.
\end{proposition}
\noindent 
This proposition holds because for each distinct part size, there are two options; either the first appearance of that part size is overlined or it is not.
One implication of this proposition is that 
 an overpartition identity follows directly from an underlying ordinary partition identity
 if the underlying identity can be refined to hold for partitions with a fixed diversity.

Since copartitions are not a set of ordinary partitions, we define the overpartition analogue of copartitions in the following way.
\begin{defn}
An $(a,b,m)$-overcopartition is a triple $(\bar\gamma,\rho,\bar\sigma)$, where $\bar\gamma$ is an overpartition with each of its parts at least $a$ and congruent to $a\pmod m$, $\bar\sigma$ is an overpartition with each of its parts at least $b$ and congruent to $b\pmod m$, and $\rho$ is an ordinary partition with the same number of parts as $\bar\sigma$, each of which have size equal to $m$ times the number of parts of $\bar\gamma$.\\
 When $a,b,m\ge 1$, we let $\overline{\cp}_{a,b,m}(n)$ denote the number of $(a,b,m)$-overcopartitions of size $n$, and we let $\overline{\CP}_{a,b,m}(n)$ denote the set of $(a,b,m)$-overcopartitions of size $n$. \\
  Also, we define the diversity of a copartition to be the sum of the diversities of the ground and the sky. That is, $dv((\gamma,\rho,\sigma)) = dv(\gamma) + dv(\sigma)$. 
  Note that this diversity is equal to the number of different row sizes in the graphical representation of the copartition $(\gamma,\rho, \sigma)$.
\end{defn}

Note that the map $\phi$ in the proof of Theorem \ref{thm:inequality} preserves diversity.
To see this, notice that in creating all of the new parts of $\tilde{\sigma}$ of some fixed size $x$, we must exhaust all of the parts of some fixed size in $\tilde{\gamma}$. 
Thus for each new part size created by $\phi$ in $\tilde{\sigma}$, exactly one part size in $\tilde{\gamma}$ vanishes.
Since $\phi$ preserves diversity, \eqref{eq:overformula} implies that 
\begin{equation}\label{overcpineq}
    \overline{\cp}_{1,1,2}^o(n) \begin{cases}
\le \overline{\cp}_{1,1,2}^e(n) & \text{if $n$ is even} \\
=\overline{\cp}_{1,1,2}^e(n) & \text{if $n$ is odd.} 
\end{cases}   
\end{equation}

Next, note that
the diversity of a $(1,1,2)$-copartition aligns with the diversity of its corresponding ${\mathcal{EO}^*}$-type partition.
To see this, rewrite $dv((\gamma,\rho,\sigma))$ as $dv(\gamma') + dv(\sigma)$,
and notice that the diversity of the ${\mathcal{EO}^*}$-type partition corresponding to $(\gamma,\rho,\sigma)$ is the same quantity. 
Thus, \eqref{overcpineq} implies Theorem \ref{thm:EOBARinequality}.

 Although we do not treat more general overcopartition functions here, we can write down the general overcopartition generating function. 
 \begin{theorem}
 Let $\overline{\cp}_{a,b,m}(r,n)$ denote the number of $(a,b,m)$-overcopartitions of size $n$ with $r$ overlined parts. Then, 
 \begin{equation}\label{thm:overcopartitions}
     \sum_{n=0}^\infty\sum_{r=0}^\infty\overline{\cp}_{a,b,m}(r,n)z^r q^n=\frac{(-zq^{b+m};q^m)_\infty}{(q^b;q^m)_\infty}
     \sum_{k=0}^{\infty} \frac{(-zq^m,q^m)_k (q^b,q^m)_k q^{ak}}{(q^m;q^m)_k(-zq^{b+m},q^m)_k}.
 \end{equation}

 \end{theorem}

\noindent Readers familiar with the notations of basic hypergeometric series will notice that the sum on the right side of \eqref{thm:overcopartitions}
is just $\displaystyle {}_2\phi_1\left({-zq^m,q^b \atop -zq^{b+m}};q^m,q^a\right)$.

 \begin{proof}
 We begin by noting that we may represent an overlined part in $\bar\gamma$ or $\bar\sigma$ by shading in the last cell of that part in their $m$-modular diagram. Furthermore, $\frac{q^{aw}(-zq^m;q^m)_w}{(q^m;q^m)_w}$ is the generating function for an overpartition into exactly $w$ parts of size $a\pmod{m}$ and $\frac{q^{bs}(-zq^m;q^m)_s}{(q^m;q^m)_s}$is the generating function for an overpartition into exactly $s$ parts of size $b\pmod{m}$, where, in both products, $z$ keeps track of the number of overlined parts. Thus, by summing over all possible dimensions of $\rho$, we can see that $$\sum_{n=0}^\infty\sum_{r=0}^\infty\overline{\cp}_{a,b,m}(r,n)z^r q^n=\sum_{w=0}^\infty\sum_{s=0}^\infty\frac{q^{msw+aw+bs}(-zq^m;q^m)_w(-zq^m;q^m)_s}{(q^m;q^m)_w(q^m;q^m)_s}.$$ Using the $q$-binomial theorem, we obtain
 \begin{align*}
 \sum_{w=0}^\infty\sum_{s=0}^\infty\frac{q^{msw+aw+bs}(-zq^m;q^m)_w(-zq^m;q^m)_s}{(q^m;q^m)_w(q^m;q^m)_s}&=\sum_{w=0}^\infty \frac{q^{aw}(-zq^m;q^m)_w}{(q^m;q^m)_w}\frac{(-zq^{m(w+1)+b};q^m)_\infty}{(q^{mw+b};q^m)_\infty}\\
 &=\frac{(-zq^b;q^m)_\infty}{(q^b;q^m)_\infty}\sum_{w=0}^\infty \frac{q^{aw}(-zq^m;q^m)_w(q^b;q^m)_w}{(q^m;q^m)_w(-zq^{b+m};q^m)_w(1+zq^b)}.
 \end{align*}
To complete the proof, we note that $$\sum_{w=0}^\infty \frac{q^{aw}(-zq^m;q^m)_w(q^b;q^m)_w}{(q^m;q^m)_w(-zq^{b+m};q^m)_w(1+zq^b)}=\left(\frac{1}{1+zq^b}\right) 
     \sum_{w=0}^{\infty} \frac{(-zq^m,q^m)_w (q^b,q^m)_w q^{aw}}{(q^m;q^m)_w(-zq^{b+m},q^m)_w}.$$
 \end{proof}

\begin{remark}
Note that, in the special case where $a=b$ and $m=2a$, we can, following the work of Chern \cite{Chern21}, define $\overline{cp}^e_{a,a,2a}(r,n)$ (resp. $\overline{cp}^o_{a,a,2a}(r,n)$) to be the number of $(a,a,2a)$-overcopartitions with an even (resp. odd) number of ground parts and then further simplify a weighted version of \eqref{thm:overcopartitions} as follows

\begin{equation}
    \sum_{n=0}^\infty \sum_{r=0}^\infty (\overline{cp}^e_{a,a,2a}(r,n)-\overline{cp}^o_{a,a,2a}(r,n))z^rq^n= \frac{(-q^{2a};q^{2a})_\infty(-zq^{2a};q^{4a})_\infty^2}{(q^{2a};q^{4a})_\infty}. \label{eq:overcopartitiondifference}
\end{equation}

\end{remark}

Since, the map $\phi$ in the proof of Theorem \ref{thm:inequality} preserves diversity, it makes sense that, when extended to overpartitions, the overpartitions in $\overline{\CP}^e_{1,1,2}(n)\setminus \phi(\overline{\CP}^o_{1,1,2}(n))$ are in bijective correspondence with overpartitions of size $n$ where all parts except the largest even part appear an even number of times. 
By the argument given at the end of Section \ref{sec:combinatorialproof}, this latter set is a combinatorial interpretation of the coefficients of the right-hand side of 
\eqref{eq:overcopartitiondifference}.

\section{A more general conjecture}\label{sec:conjecture}
Note that our combinatorial proof of Theorem \ref{thm:inequality} in Section \ref{sec:combinatorialproof} relies heavily on specific properties of $(1,1,2)$-copartitions. However, computational data suggests that Theorem \ref{thm:inequality} is a special case of a broader conjecture, which we state below.  
\begin{conj}\label{conj:positivity}
For $a,b,m \in \N$ and $n \in \N_0$,
if $b|a$, then $$\cp_{a,b,m}^o(n)\le \cp_{a,b,m}^e(n).$$

Equivalently, the $q$-series $$\frac{(-q^{a+b};q^m)_\infty}{(-q^a;q^m)_\infty(q^b;q^m)_\infty}$$ has non-negative coefficients when $b|a$. 
\end{conj}

Note that, in \cite{CraigSeaweed}, motivated by the connection to a type of Lie algebra called a seaweed algebra, Craig proved that the $q$-series $$\frac{1}{(-q^3;q^4)_\infty(q;q^4)_\infty}$$ has non-negative coefficients. Since $(-q^4;q^4)_\infty$, has obviously non-negative coefficients, Craig's result implies the special case of Conjecture \ref{conj:positivity} when $a=3$, $b=1$, and $m=4$.

Additionally, we now observe that Conjecture \ref{conj:positivity} holds when $a=b$.

\begin{theorem}
For $a, m \in \N$ and $n \in \N_0$,
$$\cp_{a,a,m}^o(n)\le \cp_{a,a,m}^e(n).$$
Equivalently, the $q$-series $$\frac{(-q^{2a};q^m)_\infty}{(-q^a;q^m)_\infty(q^a;q^m)_\infty}$$ has non-negative coefficients. 

\end{theorem}

\begin{proof}
In the special case where $a=b$, we can simplify the right side of (\ref{eq:weightedProd}) as follows:
\begin{align}
    \sum_{n=0}^\infty (\cp_{a,a,m}^e(n)-\cp_{a,a,m}^o(n))q^n&=\frac{(-q^{2a};q^m)_\infty}{(-q^a;q^m)_\infty(q^a;q^m)_\infty} \nonumber \\
    &=\frac{(-q^{2a};q^m)_\infty}{(q^{2a};q^{2m})_\infty}\label{eq:simplifiedWt}.
\end{align}
Note that, because the coefficients of (\ref{eq:simplifiedWt}) are all non-negative, we know that $\cp_{a,a,m}^e(n)\ge\cp_{a,a,m}^o(n)$ for all $n\ge 0$ and all $a,m\ge 1$. 
\end{proof}

It is still an open problem to find a combinatorial proof of \eqref{eq:simplifiedWt}. We note that our proof from Section \ref{sec:combinatorialproof} does not easily extend beyond the case $m=2a$.

It is natural to ask if finite versions of \eqref{genProduct} might also have non-negative coefficients.
Among many possible finite versions, we find that the following one seems to be well-poised and appears to have non-negative coefficients.

\begin{conj}\label{conj:finiteversion1}
For $a,b,m,N,M \in \N$, if $b|a$ and $a+b=m$, then the $q$-series $$\frac{(-q^{m};q^m)_{N+M-1}}{(-q^a;q^m)_N(q^b;q^m)_M}$$ has non-negative coefficients when $N \leq M$.
\end{conj}

When we say that this version is well-poised, we mean that 
it satisfies a nice recursion, much like the Gaussian binomial coefficients do.
If we define
$$
g_{a,b,m}(N,M;q) = g(N,M) = \frac{(-q^{m};q^m)_{N+M-1}}{(-q^a;q^m)_N(q^b;q^m)_M},
$$
then
$$
g(N,M) = g(N,M-1) + q^{mM-a} g(N-1,M).
$$
Note that Conjecture \ref{conj:finiteversion1} would imply Conjecture \ref{conj:positivity}.

\bibliographystyle{amsplain}
\bibliography{copartitions}
\end{document}